\documentclass[12pt,a4paper]{amsart}

\usepackage{graphicx}
\usepackage{color}
\usepackage{amsmath,amsthm,amsfonts,amssymb}
\usepackage{centernot}

\newtheorem{theorem}{Theorem}
\newtheorem*{Theorem}{Theorem}

\newtheorem{corollary}{Corollary}
\theoremstyle{definition}

\newtheorem{definition}{Definition}
\newtheorem*{Definition}{Definition}

\newtheorem{proposition}{Proposition}

\title{$k$-type chaos of $\mathbb{Z}^d$ actions}
\author{Anshid Aboobacker, Sharan Gopal}

\begin{document}

\begin{center}

\large{\textbf{$k$-type chaos of $\mathbb{Z}^d$ actions}}

\small{Anshid Aboobacker, Sharan Gopal}

\textit{Department of Mathematics, BITS Pilani Hyderabad Campus}

anshidaboobackerk@gmail.com, sharan@hyderabad.bits-pilani.ac.in
\end{center}

\section*{Abstract}
In this paper, we define and study the notions of $k-$type proximal pairs, $k-$type asymptotic pairs and  $k-$type Li Yorke sensitivity for dynamical systems given by $\mathbb{Z}^d$ actions on compact metric spaces. We prove the Auslander-Yorke dichotomy theorem for $k-$type notions. The preservation of some of these notions under uniform conjugacy is also studied. We also study relations between these notions and their analogous notions in the usual dynamical systems.

\noindent\textbf{Keywords:} chaos, $k-$type chaos, Li-Yorke Sensitivity, induced actions, Auslander-Yorke dichotomy theorem

\noindent\textbf{2020 MSC:} 37B05, 37C85

\section{Introduction}

A pair $(X,f)$ where $X$ is a compact metric space and $f$ is a homeomorphism on $X$ is called a \emph{dynamical system}.
The dynamics on $X$ is given by the group $\left\{ f^n : n\in \mathbb{Z}\right\}$ where $f^n = f\circ f \circ \cdots \circ f$ ($n$ times), $f^{-n} = f^{-1}\circ f^{-1} \circ \cdots \circ f^{-1}$ ($n$ times) and $f^0$ is the identity function. This dynamical system can also be seen as a $\mathbb{Z}$ action on $X$, where $(n,x)\mapsto f^n(x) $ for every $ x\in X $ and $ n\in \mathbb{Z}$. We can generalize this concept to study any group action on $X$ and in particular, the actions of $\mathbb{Z}^d$ on $X$.
Hereafter, we write a dynamical system as $(X,T)$ where $X$ is a compact metric space and $T$ is a $\mathbb{Z}^d$ action on $X$; the image of $(n,x)$ under $T$ is denoted as $T^n(x) $ for every $x\in X $ and $ n\in \mathbb{Z}^d$.

In a 2008 paper \cite{oprocha2007chain}, to study about chain recurrences in multidimensional discrete time dynamical systems, P Oprocha introduced $k-$type limit sets, $k-$type limit prolongation sets and $k-$type transitivity for $\mathbb{Z}^d$ actions where $k$ is an integer between $1$ to $2^d$. Following this idea, Sejal Shah and Ruchi Das \cite{shah2015note,shah2015different} defined $k-$type sensitivity, $k-$type periodic points, $k-$type Devaney chaos, $k-$type Li Yorke pairs, etc. They studied about preservation of many of these notions under conjugacy and uniform conjugacy. Shah and Das \cite{shah2015note} also looked into relationship between these notions in $\mathbb{Z}^d$ actions and their induced actions on hyperspace $K(X)$, the space of all compact subsets of $X$. 

Kamarudin and Dzul-Kifli \cite{kamarudin2021sufficient} considered a dynamical system $(X,f)$ and studied the induced $\mathbb{Z}^d$ actions on $X$. They showed that transitivity in the base system persists in the induced system under certain conditions.

Sensitive dependence on initial conditions (also called sensitivity) is a notion well-studied in the literature. It was introduced by Auslander and Yorke \cite{auslander1980interval} in 1980 and popularized by Devaney in 1989 with his definition of chaos. Li and Yorke \cite{li2004period} in 1975 introduced the concept of Li-Yorke pairs while studying the interval maps. Akin and Kolyada \cite{akin2003li} introduced Li-Yorke sensitivity combining both these notions.

In this paper, we define and study $k-$type proximal pairs, $k-$type asymptotic pairs, $k-$type Li Yorke Sensitivity, $k-$type Li Yorke pairs and $k-$type Li-Yorke chaos and various relations between them. We also study that all these notions are preserved under conjugacies. Finally, we look into how these notions work in the induced $\mathbb{Z}^d$ actions.

\section{Preliminaries}

We will list some definitions and results for $\mathbb{Z}$ actions first and then discuss the corresponding definitions and results for $k-$type $\mathbb{Z}^d$ actions. Throughout the paper, $X$ denotes a compact metric space with metric $d$, $f:X\longrightarrow X $ is a homeomorphism and $T$ is a $\mathbb{Z}^d$ action on $X$.

A point $x\in X$ is called a \emph{periodic point} in the system $(X,f)$, if $f^n(x)=x$ for some $n \in \mathbb{N}$. For $x\in X$, we define the \emph{forward orbit} of $x$ as $\{ f^n(x) | n \geq 0 \}$.
We say that $f$ is \emph{topologically transitive} if there is a point $x\in X$ whose forward orbit is dense in $X$ and such an $x$ is called a transitive point.

The map $f$ is said to have \emph{sensitive dependence on initial conditions} or \emph{sensitivity} if there is an $r > 0$ (independent of the point) such that for each point $x\in X$ and for each $\epsilon> 0$ there is a point  $y\in X$ with $d(x, y) < \epsilon$ and a $k > 0$ such that $d(f^k(x), f^k(y)) > r$. 

A dynamical system $(X, f)$ is said to be \emph{ Auslander-Yorke chaotic} if $f$ has sensitive dependence on initial conditions and is topologically transitive. A dynamical system $(X, f)$ is said to be \emph{Devaney chaotic} if $f$ has sensitive dependence on initial conditions, is topologically transitive and the set of periodic points is dense in $X$. 

Two points $x,y\in X$ are said to form a \emph{proximal pair} if $\liminf\limits_{n\rightarrow \infty} d(f^n(x),f^n(y)) = 0$, i.e., if there exists an increasing sequence $n_k\in \mathbb{N}$, such that $\lim\limits_{n_k\rightarrow \infty} d(f^{n_k}(x),f^{n_k}(y)) = 0$. We also say that \emph{$y$ is proximal to $x$} and vice versa. 
If for every increasing sequence $n_k\in \mathbb{N}$,  $\lim\limits_{n_k\rightarrow \infty} d(f^{n_k}(x),f^{n_k}(y)) = 0$ then $x,y$ are said to form an \emph{asymptotic pair}; here again, $y$ is said to be \emph{asymptotic to} $x$ and vice versa. 

The set of all proximal pairs of $f$ is denoted by $Prox(f)$ and the set of all elements in $X$ proximal to a point $x\in X$ is called the \emph{proximal cell of $x$}, denoted by $Prox(f)(x)$. The set of all asymptotic pairs of $f$ is denoted by $Asym(f)$ and the set of all elements asymptotic to a point $x\in X$ is called as \emph{asymptotic cell of x} denoted by $Asym(f)(x)$. We have an equivalent definition for the set of all proximal pairs and the set of all asymptotic pairs, as given below.

For an $\epsilon>0$, define 
\[ 
V_{\epsilon} = \{ (x,y)\in X\times X : d(x,y) < \epsilon \}, \]
\[
\overline{V_{\epsilon}} = \{ (x,y)\in X\times X : d(x,y) \leq \epsilon \}. \]

Then, 
\[Prox(f) = \bigcap_{\epsilon>0}\bigcup_{\substack{n> 0 \\ n\in \mathbb{N}}} \left[ f^{-n}\times f^{-n} (V_{\epsilon}) \right] \]

\[ Asym_{\epsilon}(f)= \bigcup_{\substack{n>0 \\ n\in \mathbb{N}}} f^{-n}\times f^{-n} \left( \bigcap_{\substack{m>0 \\ m\in \mathbb{N}}} [ f^{-m}\times f^{-m} (\overline{V_{\epsilon}})] \right) \]

and \[ Asym(f) = \bigcap_{\epsilon>0} Asym_{\epsilon}(f).   \]
 
Two points $x,y \in X$ which form a proximal pair, but not an asymptotic pair are said to form a \emph{Li-Yorke pair}, i.e., if $\liminf\limits_{n\rightarrow \infty} d(f^n(x),f^n(y)) = 0$ and  $\limsup\limits_{n\rightarrow \infty} d(f^n(x),f^n(y)) > 0$. The set of all Li-Yorke pairs of $f$ is denoted by $LY(f)$.

 A subset $S\subset X$ is called a \emph{scrambled} set if any two distinct points in $S$ form a Li-Yorke pair. $(X, f)$ is said to be \emph{Li-Yorke chaotic} if there exists an uncountable scrambled set in $X$.

 Combining the concepts of Li Yorke pair and sensitivity, Akin and Kolyada in 2003 \cite{akin2003li} defined the concept of \emph{Li Yorke sensitivity}. A dynamical system $(X,f)$ is said to be Li Yorke sensitive if there is an $\epsilon>0$ such that every $x\in X$ is a limit point of $Prox(f)(x) \setminus Asym_{\epsilon}(f)(x)$; in other words, $x\in \overline{Prox(f)(x) \setminus Asym_{\epsilon}(f)(x)}$.

A dynamical system $(X,f)$ is called \emph{equicontinuous} if for every $\epsilon>0$ there exists some $\delta>0$ such that, for every $x,y\in X$ with $d(x,y)<\delta$, we have  $d(f^n(x),f^n(y))<\epsilon $ for every $ n  \in \mathbb{Z}$; i.e., the family of maps $\{f^n: n\in \mathbb{Z}\}$ is uniformly equicontinuous. A point $x\in X$ is called an \emph{equicontinuous point} if for any $\epsilon>0$ there exists some $\delta>0$ such that for every $y\in X$ with $d(x,y)<\delta$, we have  $d(f^n(x),f^n(y))<\epsilon$ for every $ n  \in \mathbb{Z}$. The set of all equicontinuous points in $(X,f)$ is denoted by $Eq(f)$. 

Finally, a transitive system is called \emph{almost equicontinuous} if there exists an equicontinuous point in it. A dynamical system is \emph{minimal} if it does not contain any proper subsystem. We have the following theorem (see \cite{akin1996transitive}), popularly known as the \textit{Auslander-Yorke Dichotomy theorem}, showing the dichotomy of equicontinuous systems and sensitive systems.

\begin{Theorem}

Let $(X, f)$ be topologically transitive. 

If $(X, f)$ is almost equicontinuous, then the set of equicontinuous points coincides with the set of transitive points (and so the set of equicontinuous points is a dense $G_{\delta}$).
In particular, a minimal almost equicontinuous dynamical system is equicontinuous. 

If $(X, f)$ has no equicontinuous points then it is sensitive. In particular a minimal system is either equicontinuous or sensitive.
\end{Theorem}

\vspace{0.5cm}

Now, for a $\mathbb{Z}^d$ action, we have the following definitions given by P Oprocha\cite{oprocha2007chain}, S Shah and R Das \cite{shah2015note,shah2015different}. Let $T: \mathbb{Z}^d \times X \rightarrow X$ be a $\mathbb{Z}^d$ action on $X$ and let $k\in \{1,2,3,\dots,2^d\}$. Let $k^b$ represent $k-1$ in $d$-positional binary system, i.e., $k-1= \sum_{i=1}^d k_i^b 2^{i-1}$ with $k^b\in \{0,1\}^d$. For $x,y\in \mathbb{Z}^d $, we say $x>^ky$ if $(-1)^{k_i^b} x_i > (-1)^{k_i^b} y_i ~~\forall i$, where $x=(x_1,\dots,x_d)$ and $y=(y_1,\dots,y_d)$.

 A point $x\in X $ is a \emph{$k-$type periodic point} if there is an ${n}\in \mathbb{Z}^d$, ${n}>^k{0}$ satisfying $T^{n}(x)=x$. A subset $A$ of $X$ is said to be positively invariant if $T^{n}(A)\subset A$ for any ${n}>^k{0}$ and negatively invariant if $T^{n}(A)\subset A$ for any ${n}<^k{0}$.
 
 $T$ is \emph{$k-$type transitive} if there exists $ x\in X$ such that its forward orbit $Orb_T(x)=\{ T^{n}(x): {n}\in \mathbb{Z}^d, {n}>^k{0}\}$ is dense in X. If $X$ does not have any isolated points, it is equivalent to saying that for any two nonempty open sets $U, V$ there exists ${n}>^k{0}$ such that $T^{n}(U)\bigcap V \neq \emptyset$. A point $x\in X$ whose forward orbit is dense in $X$ is called a \emph{$k-$ type transitivity point}.

$T$ has \emph{$k-$type sensitivity} if there exists $\delta >0$ such that for any $x\in X$ and any $\epsilon > 0$ there exists $y\in B_d(x,\epsilon)$ and ${n}>^k{0}$ such that $d(T^{n}(x), T^{n}(y)) > \delta$, or equivalently, if there exists $\delta >0$ such that for any non empty open set $U\subset X$, there exists $x,y\in U$ and ${n}>^k{0}$ such that $d(T^{n}(x), T^{n}(y)) > \delta$.

We say that $T$ is \emph{$k-$ type Devaney chaotic} if $T$ is $k-$type transitive, has a dense set of $k-$type periodic points and has $k-$type sensitivity.
We say that $T$ is \emph{$k-$type Auslander-Yorke chaotic} if $T$ is $k-$type transitive and has $k-$type sensitivity.

The \emph{$k-$type limit set} of a point $x\in X$ is $L^k(x)=\{y\in X:$ there exists a sequence $({t_s})_{s\in \mathbb{N}} $ in $\mathbb{Z}^d$ with $ {t_{s+1}} >^k {t_{s}}  $ such that $\lim_{s\rightarrow\infty} T^{{t_s}}(x)=y\}$. 
The \emph{$k-$type limit prolongation set} of a point $x\in X$ is given by $J^k(x)=\{y\in X:$ there exists sequence $(x_s)_{s\in \mathbb{N}}$ in $X$ and $({t_s})_{s\in \mathbb{N}} $ in $\mathbb{Z}^d$ with $ {t_{s+1}} >^k {t_{s}}  $ such that $\lim_{s\rightarrow\infty} x_s=x$ and $\lim_{s\rightarrow\infty} T^{{t_s}}(x_s)=y\}$. Equivalently, we can say $y\in J^k(x)$ if for every neighbourhood $U$ of $y$, neighbourhood $V$ of $x$ and $ {N}>^k{0}$ there is an ${{n}}>^k{N}$ and a point $z\in V$ such that $T^{{n}}(z)\in U$.

Kamarudin and Dzul-Kifli in their paper \cite{kamarudin2021sufficient} defined and studied a special class of induced dynamical systems. When discussing induced systems, we use bold formatting for elements of $\mathbb{Z}^d$ to distinguish them from elements of $\mathbb{Z}$. Let $(X,f)$ be a dynamical system with $f:X\rightarrow X$ a homeomorphism on $X$. Let $r:\mathbb{Z}^d \rightarrow \mathbb{Z} $ be a homomorphism defined by $r(\textbf{n}) = h_1n_1 + h_2n_2 + \dots + h_dn_d$, where $h_i \in \mathbb{Z} $ for all $i$. Define $T_f: \mathbb{Z}^d \times X \rightarrow X$ by $T_f(\textbf{n},x)= f^{r(\textbf{n})}(x) $. Then $T_f$ is a $\mathbb{Z}^d$ action on $X$ induced by $f$. They proved the following theorem.

\begin{Theorem}
Let $f:X\rightarrow X$ be a transitive continuous map and $T_f: \mathbb{Z}^d \times X \rightarrow X$ be a $\mathbb{Z}^d$ action on $X$ induced by $f$. Suppose $k\in \{1,2,3,\dots,2^d\}$. If there exists $\textbf{m}>^k\textbf{0}$ such that $r(\textbf{m})=1$, then $T_f$ is $k-$type transitive. 
\end{Theorem}

\vspace{1cm}
\section{Results}

\subsection{Auslander-Yorke Dichotomy theorem}

In this section, we state and prove the Auslander-Yorke Dichotomy theorem for $k-$type dynamical systems. We first frame the required definitions and prove some preliminary results. As usual, we denote by $(X, T)$, a dynamical system where $T: \mathbb{Z}^d \times X \rightarrow X$ is a $\mathbb{Z}^d$ action on $X$.

\begin{definition}
    $T$ is called \emph{equicontinuous} if for every $\epsilon>0$ there exists some $\delta>0$ such that for every $x,y\in X$ with $d(x,y)<\delta$, we have $d(T^{n}(x),T^{n}(y))<\epsilon $ for all ${n}  \in \mathbb{Z}^d$, i.e., the family of maps $\{T^{n}: {n}\in \mathbb{Z}^d\}$ is uniformly equicontinuous.

    A point $x\in X$ is called an \emph{equicontinuous point} if for any $\epsilon>0$ there exists some $\delta>0$ such that for every $y\in X$ with $d(x,y)<\delta$, we get $d(T^{n}(x),T^{n}(y))<\epsilon $ for all $ {n}  \in \mathbb{Z}^d$. The set of all equicontinuous points in $(X, T)$ is denoted by $Eq(T)$. 
    A $k-$type transitive system is called \emph{almost equicontinuous} if there exists an equicontinuous point in it.
\end{definition}

\begin{definition}
    For $l\in\mathbb{N}$, define $G_l=\{x\in X:$ there is a neighbourhood $U$ of $x$ such that $x_1, x_2 \in U$ implies that $d(T^{n}(x_1),T^{n}(x_2))\leq \frac{1}{l} $ for all $ {n} \in \mathbb{Z}^d\}$
\end{definition}

Note that $G_l$ is an open negatively invariant subset of $X$ and they form a decreasing sequence of sets, i.e., $G_1\supset G_2 \supset \cdots $.

\begin{proposition}
    $\bigcap\limits_{l\in\mathbb{N}} G_l $ $= Eq(T)$.
\end{proposition}
\begin{proof}

    Let $x\in\bigcap\limits_{l\in\mathbb{N}} G_l$ and $\epsilon>0$. Choose $l\in\mathbb{N}$ such that $\frac{1}{l}<\epsilon$. Since $x\in G_l$, there exists a neighbourhood $U$ of $x$ such that $x_1, x_2 \in U$ implies that $d(T^{{n}}(x_1),T^{{n}}(x_2))\leq \frac{1}{l} $ for all $ {n} \in \mathbb{Z}^d$. Choose a $\delta>0$ such that $B_d(x,\delta) \subset U$. Then, for every $ y\in X$ with $d(x,y)<\delta$, we get $d(T^{n}(x),T^{n}(y))<\epsilon $  for all ${n}  \in \mathbb{Z}^d$. Hence, $x\in Eq(T)$.

    Conversely, let $x\in Eq(T)$ and $l\in\mathbb{N}$. Choosing $\epsilon=\frac{1}{2l}$, we get $\delta>0$ such that, for every $ y\in X$,  $d(x,y)<\delta$ implies that $d(T^{n}(x),T^{n}(y))<\frac{1}{2l} $ for all $ {n}  \in \mathbb{Z}^d$. Take $U=B_d(x,\delta)$. Then, for every $ x_1, x_2 \in U$, we have $d(T^{n}(x_1),T^{n}(x))\leq \frac{1}{2l} $ and $d(T^{n}(x),T^{n}(x_2))\leq \frac{1}{2l} $  for all $ n \in \mathbb{Z}^d$ and thus $d(T^{n}(x_1),T^{n}(x_2))\leq \frac{1}{l} $, i.e., $x\in G_l$. Hence $x\in\bigcap\limits_{l\in\mathbb{N}} G_l$.
    
\end{proof}

\begin{proposition}
    If $(X,T)$ is $k-$type transitive, then $J^k(x)=X$ for every $x\in X$.
\end{proposition}
\begin{proof}
Let $x\in X$. Let $y\in X$, $U$ be a neighbourhood of $y$, $V$ be a neighbourhood of $x$ and ${n}>^k{0}$. Consider the set $U\setminus\{T^{{m}}(x): {0}<^k {m} \leq^k {N}\}$. This is open and non-empty. So there exists an $ {n}>^k{N}$ such that $T^{{n}}(x) \in U\setminus\{T^{{m}}(x): {0}<^k {m} \leq^k {N}\}$. This implies that $y\in J^k(x)$ and thus $J^k(x)=X$ for every $x\in X$.
\end{proof}

\begin{proposition}
    If $x$ is an equicontinuous point in $(X,T)$, then $L^k(x)=J^k(x)$.
\end{proposition}
\begin{proof}
    Clearly $L^k(x)\subseteq J^k(x)$. To show that $J^k(x)\subseteq L^k(x)$, let $y\in J^k(x)$. Let $U$ be any neighbourhood of $y$ and let ${n}>^k{0}$. Choose $\epsilon>0$ such that $B(y,\epsilon)\subset U$. 

    Since $x$ is an equicontinuous point, for every $\epsilon>0$, there is a $\delta>0 $ such that for any $ z\in X$, $d(z,x)<\delta$ implies that $d(T^{{n}}(z),T^{{n}}(x))< \frac{\epsilon}{2} $ for all $ {n}>^k{0}$.

    Since $y\in J^k(x)$, for every neighbourhood $ U$ of $y$, neighbourhood $V$ of $x$ and $ {N}>^k{0}$ there is an ${{n}}>^k{N}$ and a point $z\in V$ such that $T^{{n}}(z)\in U$. Taking $U=B(y,\epsilon/2)$ and $V=B(x,\delta)$, we have for any ${N}>^k{0}$, an $ {n}>^k{N}$ and $z\in B(x,\delta)$ such that $T^{n}(z)\in B(y,\epsilon/2) $, ie, $d(T^{n}(z),y)<\frac{\epsilon}{2}$. 

    So, $d(T^{{n}}(x),y)\leq d(T^{{n}}(x),T^{{n}}(z)) + d(T^{{n}}(z),y) <\epsilon$ and $T^{{n}}(x) \in B(y,\epsilon) \subset U$. This implies that $y\in L^k(x)$ and $J^k(x)\subset L^k(x)$; hence $L^k(x)=J^k(x)$.\end{proof}

\begin{theorem}
    
    Let $(X, T)$ be $k-$type transitive. 
    
    If $(X, T)$ is almost equicontinuous then the set of equicontinuous points coincides with the set of $k-$type transitive points (and so the set of equicontinuous points is a dense $G_{\delta}$).
    
    In particular, a minimal almost equicontinuous dynamical system is equicontinuous. $(X, T)$ has no equicontinuous points if and only if it is $k-$type sensitive. In particular, a minimal $\mathbb{Z}^d$ action is either equicontinuous or $k-$type sensitive.
\end{theorem}
\begin{proof}
    Let $(X,T)$ be an almost equicontinuous system, i.e., $ Eq(T)=\bigcap\limits_{l\in\mathbb{N}} G_l \neq \emptyset$ which implies that for all $l\in \mathbb{N}, ~ G_l\neq \emptyset$. Let $x$ be a $k-$type transitive point. Since $G_l$ is an open set, there exists {$n_l$}$>^k0$ such that $T^{n_l}(x)\in G_l$ and thus $ x\in G_l$, as $G_l$ is negatively invariant. Hence $x \in \bigcap\limits_{l\in\mathbb{N}} G_l= Eq(T)$.
    Conversely, if $(X,T)$ is $k-$type transitive then for every $x\in X, ~J^k(x)=X$. If $x$ is an equicontinuous point then $L^k(x)=J^k(x)$. Thus $L^k(x)=X$ and $x$ is a $k-$type transitive point.

    If $(X,T)$ has no equicontinuous points, then $\bigcap G_l$ is empty and there is an $l\in \mathbb{N}$ such that $G_l = \emptyset$. So there exists $\delta=\frac{1}{l} >0$ such that for any non empty open set $U\subset X$, there are points $ x_1,x_2\in U$ such that $d(T^{{n}}(x_1),T^{{n}}(x_2)) > \delta$ for some ${n}>^k{0}$. Hence $(X,T)$ is $k-$type sensitive.

    Conversely, let $(X,T)$ be $k-$type sensitive. Assume that there is an equicontinuous point in $X$, say $x$. Then $x\in\bigcap\limits_{l\in\mathbb{N}} G_l$, i.e., for every $l \in \mathbb{N},$ there is a neighbourhood $U_l$ of $x$ such that for any two points $x_1,x_2\in U_l, ~d(T^{{n}}(x_1),T^{{n}}(x_2)) \leq \frac{1}{l} $ for all $ {n}>^k{0}$. This implies that $(X, T)$ is not $k-$type sensitive, which is a contradiction. Hence $(X, T)$ has no equicontinuous points.
\end{proof}

\vspace{1cm}

\subsection{$k-$type Li Yorke Sensitivity}

\begin{definition} Let $(X,T)$ be a dynamical system given by a $\mathbb{Z}^d$ action $T$ on $X$. The set of all \emph{$k-$type proximal pairs} of $T$ is defined as

\[k-Prox(T) = \bigcap_{\epsilon>0}\bigcup_{\substack{{n}>^k {0} \\ {n}\in \mathbb{Z}^d}} \left[ T^{{-n}}\times T^{{-n}} (V_{\epsilon}) \right]. \]
Any pair $(x,y) \in k-Prox(T) $ is called a $k-$type proximal pair of $T$. $k$-type proximal cell of $x\in X$ is $k-Prox(T)(x)= \{ y\in X | (x,y) \in k-Prox(T) \} $.
\end{definition} 

\begin{proposition}
    $(x,y) \in k-Prox(T)$ if and only if there exists a sequence $({t_s})_{s\in \mathbb{N}} $ in $\mathbb{Z}^d$ with $ {t_{s+1}} >^k {t_{s}}  $ such that $d(T^{{t_s}}(x) , T^{{t_s}}(y)) \rightarrow 0 $ as $s\rightarrow \infty$.
\end{proposition}
\begin{proof}
    Let $(x,y) \in k-Prox(T)$, i.e., for every $\epsilon >0$, there is an ${n}>^k{0}$ such that $(x,y)\in T^{{-n}}\times T^{{-n}} (V_{\epsilon})$.  
    For $\epsilon=\frac{1}{s}$, $s\in \mathbb{N}$ choose ${t_s} \in\mathbb{Z}^d$ such that  $ {t_{s+1}} >^k {t_{s}} $ and $(x,y)\in T^{{-t_s}}\times T^{{-t_s}} (V_{\epsilon})$. 
    This implies that there exists a sequence $({t_s})_{s\in \mathbb{N}} $ in $\mathbb{Z}^d$ with $ {t_{s+1}} >^k {t_{s}}  $ such that $d(T^{{t_s}}(x) , T^{{t_s}}(y)) \rightarrow 0 $ as $s\rightarrow \infty$.

    Conversely, suppose there exists a sequence $({t_s})_{s\in \mathbb{N}} $ in $\mathbb{Z}^d$ with $ {t_{s+1}} >^k {t_{s}}  $ such that $d(T^{{t_s}}(x) , T^{{t_s}}(y)) \rightarrow 0 $ as $s\rightarrow \infty$. Then for every $\epsilon >0$, there is a ${t_s} >^k{0}$ such that $d(T^{{t_s}}(x) , T^{{t_s}}(y)) < \epsilon$ i.e., $(x,y) \in T^{{-t_s}}\times T^{{-t_s}} (V_{\epsilon})$.  Hence $(x,y) \in k-Prox(T)$.
\end{proof}

\begin{definition} For any $\epsilon>0$, \[ k-Asym_{\epsilon}(T)= \bigcup_{\substack{{n}>^k {0} \\ {n}\in \mathbb{Z}^d}} T^{{-n}}\times T^{{-n}} \left( \bigcap_{\substack{{m}>^k {0} \\ {m}\in \mathbb{Z}^d}} [ T^{{-m}}\times T^{{-m}} (\overline{V_{\epsilon}})] \right) \]

and  \[ k-Asym(T) = \bigcap_{\epsilon>0} k-Asym_{\epsilon}(T).   \]
We call elements $(x,y)$ in $k-Asym(T)$ as \emph{$k-$type asymptotic pairs} of $T$. $k$-type $\epsilon-$asymptotic cell of $x\in X$ is $k-Asym_{\epsilon}(T)(x)= \{ y\in X | (x,y) \in k-Asym_{\epsilon}(T) \} $ and $k$-type asymptotic cell of $x\in X$ is $k-Asym(T)(x)= \{ y\in X | (x,y) \in k-Asym(T) \}. $

\end{definition}

\begin{definition}
    $(X,T)$ has $k-$type Li-Yorke Sensitivity if there is an $\epsilon>0$ such that for every $x\in X$, $x \in \overline{k-Prox(T)(x) \setminus k-Asym_{\epsilon}(T)(x)}$.
\end{definition}

\begin{theorem} \label{equivalences}
    The following are equivalent for a dynamical system $(X,T)$:
    \begin{enumerate}
        \item $(X,T)$ is $k-$type sensitive.
        \item $\exists \delta>0$ such that $k-Asym_{\delta}(T)$ is a first category subset of $X\times X$.
        \item $\exists \delta>0$ such that $\forall x\in X, k-Asym_{\delta}(T)(x)$ is a first category subset of $X$.
        \item $\exists \delta>0$ such that $\forall x\in X,$ $x\in \overline{X \setminus k-Asym_{\delta}(T)(x)}$.
        
    \end{enumerate}
\end{theorem}
\begin{proof}
    Define $k-C_{{n},\delta} = T^{{-n}}\times T^{{-n}} \left( \bigcap_{\substack{{m}>^k {0} \\ {m}\in \mathbb{Z}^d}} [ T^{{-m}}\times T^{{-m}} (\overline{V_{\delta}})] \right) $. \vspace{0.3cm} Then $(x,y)$ lies in the closed set $k-C_{{n},\delta}$ if and only if $d(T^{{i}}(x),T^{{i}}(y))\leq \delta $ for all $ {i}\geq ^k {n}$.
    So, we have $k-Asym_{\delta}(T)=\bigcup_{\substack{{m}>^k {0} \\ {m}\in \mathbb{Z}^d}} k-C_{{n},\delta}$ and $k-Asym_{\delta}(T)(x)=\bigcup_{\substack{{m}>^k {0} \\ {m}\in \mathbb{Z}^d}} k-C_{{n},\delta}(x)$. \vspace{0.3cm}

    \textbf{(Not 1 $\implies$ Not 2)}
    $(X,T)$ is not $k$- type sensitive, i.e., for all $\delta>0$, there exists $x\in X$ and $\epsilon>0$ such that for every $y\in B_{d}(x,\epsilon)$ and ${n}>^k{0}$, $d(T^{{n}}(x),T^{{n}}(y))\leq \frac{\delta}{2}$. 
    Thus, for every $\delta>0$, there exists an open non empty subset $U$ of $X$ such that $U\times U \subset \bigcap_{{n}>^k{0}} T^{{-n}}\times T^{{-n}} (\overline{V_{\delta}}) \subset k-Asym_{\delta}(T)$. Hence $k-Asym_{\delta}(T)$ is not a first category set for all $\delta >0$.

    \textbf{(Not 1 $\implies$ Not 4)}
    $(X,T)$ is not $k$- type sensitive, i.e., for every $\delta>0$ there exists an open non-empty subset $U$ of $X$ such that $U \subset k-Asym_{\delta}(T)(x)$ for $x\in U$. Thus for every $\delta>0$ there exists $x\in X$ and an open neighbourhood $U$ of $x$ such that $U \bigcap (X\setminus k-Asym_{\delta}(T)(x)) = \emptyset$. This implies that for all $\delta>0$ there exists $x\in X$ such that $x\not\in \overline{X \setminus k-Asym_{\delta}(T)(x)}$.

    \textbf{(Not 2 $\implies$ Not 3)}
    If $k-Asym_{\delta}(T)$ is not of the first category, then by Baire Category Theorem, there is an ${n}$ such that $k-C_{{n},\delta}$ has non-empty interior. In other words, there exist non-empty open sets $U, V$ such that $U\times V \subset k-C_{{n},\delta}$. If $x\in U$, then $V \subset k-C_{{n},\delta}(x) \subset k-Asym_{\delta}(T)(x)$. Hence $ k-Asym_{\delta}(T)(x)$ is not of the first category.

    \textbf{(Not 3 $\implies$ Not 1)}
     If $k-Asym_{\delta}(T)(x)$ is not of first category, then by Baire Category Theorem, there exists ${n}>^k{0}$ and an open non-empty set $U$ such that $U \subset k-C_{{n},\delta}(x)$. By triangle inequality, $U\times U \subset k-C_{{n},2\delta}$. By continuity of $T^{{m}}$, for ${0}\leq^k {m} \leq^k {n}$, we can shrink $U$ to a non empty open set $V$ such that $d(T^{{i}}(x),T^{{i}}(y))\leq 2\delta $ for all $ {i}>^k {0}$ for $x,y\in V$. Hence $(X,T)$ is not $k$- type sensitive.

     \textbf{(3 $\implies$ 4)} 
     Given that there exists $\delta>0$ such that for all $x\in X,$ $ k-Asym_{\delta}(T)(x)$ is a first category subset of $X$. Take any open neighbourhood $U$ of $x$, then $U\not\subset k-Asym_{\delta}(T)(x)$ which implies that $U $ intersects $ (X\setminus k-Asym_{\delta}(T)(x))$. Hence $x\in \overline{X \setminus k-Asym_{\delta}(T)(x)}$. \end{proof}

\begin{proposition} \label{sensitive-LYsensitive}
If $(X,T)$ is $k-$type Li Yorke sensitive then it is $k-$type sensitive. If $(X,T)$ is $k-$type  sensitive and for every $x\in X$,  $\overline{k-Prox(T)(x)}=X$ then it is $k-$type Li Yorke sensitive.    
\end{proposition}
\begin{proof}
    $(X,T)$ is $k-$type Li Yorke sensitive 
    
\hspace{2cm}    $\implies$ $\exists\epsilon>0$ such that $\forall x\in X$, $x \in \overline{k-Prox(T)(x) \setminus k-Asym_{\epsilon}(T)(x)}$ 
    
\hspace{2cm}    $\implies$ $\exists \epsilon>0$ such that $\forall x\in X,$ $x\in \overline{X \setminus k-Asym_{\epsilon}(T)(x)}$ 
    
\hspace{2cm}    $\implies$ $(X,T)$ is $k-$type sensitive (by Theorem \ref{equivalences}). 

    Now, when $(X,T)$ is $k-$type sensitive, there exists an $\epsilon>0$ such that for every $x\in X, ~k-Asym_{\epsilon}(T)(x)$ is a first category subset of $X$ (by Theorem \ref{equivalences}). This implies $k-Asym_{\epsilon}(T)(x)$ is nowhere dense. Since for every $x\in X,$ $\overline{k-Prox(T)(x)}=X$, we get that for all $x \in X,$ ${k-Prox(T)(x) \setminus k-Asym_{\epsilon}(T)(x)}$ is dense in $X$. Hence $(X, T)$ is $k-$type Li Yorke sensitive.
\end{proof}

\begin{proposition}
    If $(X, T)$ is $k-$type  sensitive and has a fixed point which is the unique minimal subset, then it is  $k-$type Li Yorke sensitive.
\end{proposition}
\begin{proof}
    Let $u\in X$ be the fixed point, which is the unique minimal subset of $X$. Then $(u,u)$ is a fixed point in $X\times X$ and is the unique minimal subset of $X\times X$. Consider any two points $x,y \in X$. 
    The set $\overline{\{(T^{{n}}(x),T^{{n}}(y)): {n}>^k{0}\}}$ is an invariant closed set and contains a minimal set. 
    Then $(u,u) \in \overline{\{(T^{{n}}(x),T^{{n}}(y)): {n}>^k{0}\}}$,  
    which implies that there is an increasing sequence $(${$t_s$}) 
    $\subset \mathbb{Z}^d$ such that $T^{{t_s}}(x)\longrightarrow u$ and $T^{{t_s}}(y)\longrightarrow u$, 
    i.e., $d(T^{{t_s}}(x),T^{{t_s}}(y)) \longrightarrow 0$ as $s\rightarrow \infty$. So $(x,y)\in k-Prox(T)$ 
    and $ k-Prox(T)=X\times X$. Thus for every $x\in X, ~k-Prox(T)(x)=X$. Hence $(X,T)$ is  $k-$type Li Yorke sensitive by Proposition \ref{sensitive-LYsensitive}.
\end{proof}

\begin{proposition}
    If the system $(X, T)$ is $k-$type Li Yorke sensitive (or $k-$type sensitive), then the product system $(X\times Y, T\times S)$ is $k-$type Li Yorke sensitive (or $k-$type sensitive) for any dynamical system $(Y, S)$.
\end{proposition}
\begin{proof}
Taking the supremum metric on the product space $X \times Y$, it is easy to see that for any $x_1,x_2\in X$ and $y \in Y$, $(x_1,x_2)\in k-Prox(T) $ if and only if $ ((x_1,y),(x_2,y))\in k-Prox(T\times S)$ and $(x_1,x_2)\in k-Asym_{\delta}(T) $ if and only if $ ((x_1,y),(x_2,y))\in k-Asym_{\delta}(T\times S)$.

If $(X,T)$ is $k-$type Li Yorke sensitive, then there exists $\delta>0$ such that for every $x\in X$, there is a sequence $(x_n) \in k-Prox(T)(x)\setminus k-Asym_{\delta}(T)(x)$ such that $(x_n)\rightarrow x$. Take the same $\delta$ and consider any $(x,y)\in X\times Y$; then there exists sequence $(x_n,y_n)$ where $y_n=y$ for all $n\in \mathbb{N}$ such that $(x_n,y_n)\rightarrow (x,y)$ and $(x_n,y_n) \in k-Prox(T\times S)(x,y)\setminus k-Asym_{\delta}(T\times S)(x,y)$. Hence $(X\times Y, T\times S)$ is $k-$type Li Yorke sensitive. 

If $(X,T)$ is $k-$type sensitive, then there exists $\delta>0$ such that for every $x\in X$ and $\epsilon>0$, there is a $z\in B_d(x,\epsilon)$ and ${n}>^k{0}$ such that $d(T^{{n}}(x),T^{{n}}(z)) > \delta$. For $(X\times Y, T\times S)$, take the same $\delta>0$. If $(x,y)\in X\times Y$ and $\epsilon>0$, then there is $(z,y)\in B_d((x,y),\epsilon)$ and the same ${n}>^k{0}$ such that $d((T\times S)^{{n}}(x,y),(T\times S)^{{n}}(z,y)) > \delta$. Hence $(X\times Y, T\times S)$ is $k-$type sensitive. 
\end{proof}

\vspace{.5cm}

\subsection{Preservation under conjugacy}

\begin{Definition}
    Let $T_1 : \mathbb{Z}^d \times X \rightarrow X$ and $T_2 : \mathbb{Z}^d \times Y \rightarrow Y$ be $\mathbb{Z}^d$-actions on $X$ and $Y$ respectively. 
    If there exists a homeomorphism $h : X \rightarrow Y$ such that $h\circ T^{{n}}_1 = T^{{n}}_2\circ h,$ for every ${n}\in\mathbb{Z}^d$, then $h$ is called a \emph{conjugacy} from $(X,T_1)$ to $(Y,T_2)$.
\end{Definition}

Note that $h$ and $h^{-1}$ are uniformly continuous maps as $X$ and $Y$ are compact.

\begin{proposition}
    Let $h$ be a conjugacy from $(X,T_1)$ to $(Y,T_2)$. 
    
    (i) $(x_1,x_2)\in k-Prox(T_1)$ if and only if $(h(x_1),h(x_2)) \in k-Prox(T_2)$.
    
    (ii) $(x_1,x_2)\in k-Asym(T_1)$ if and only if $(h(x_1),h(x_2)) \in k-Asym(T_2)$.

    (iii) If there is an $\epsilon>0 $ such that $ y\not\in k-Asym_{\epsilon}(T_1)(x)$, then there exists a $\delta>0$ such that $ h(y)\not\in k-Asym_{\delta}(T_2)(h(x))$.
    
\end{proposition}
\begin{proof}
   \textbf{(i)} $(x_1,x_2)\in k-Prox(T_1)$
       
\hspace{2cm}$\implies (x_1,x_2) \in \bigcap_{\epsilon>0}\bigcup_{\substack{{n}>^k {0} \\ {n}\in \mathbb{Z}^d}}  T_1^{{-n}}\times T_1^{{-n}} (V_{\epsilon})$ \vspace{.3cm}
     
\hspace{2cm}$\implies \forall \epsilon>0 ~\exists {n}>^k{0} $ such that $d(T_1^{{n}}(x_1),T_1^{{n}}(x_2)) <  \epsilon. $
     
By uniform continuity of $h$, for every $\delta>0$ there is an $\epsilon>0$ such that $d(T_1^{{n}}(x_1),T_1^{{n}}(x_2)) <  \epsilon$ implies $d(h\circ T_1^{{n}}(x_1), h\circ T_1^{{n}}(x_2)) < \delta$.
     
\hspace{2cm}$  \implies \forall \delta>0 ~\exists {n}>^k{0} $ such that $d(h\circ T_1^{{n}}(x_1), h\circ T_1^{{n}}(x_2)) < \delta $
     
\hspace{2cm}$ \implies \forall \delta>0 ~\exists {n}>^k{0} $ such that $d(T_2^{{n}}\circ h (x_1), T_2^{{n}}\circ h(x_2)) < \delta$ 
     
\hspace{2cm}$\implies (h(x_1),h(x_2)) \in \bigcap_{\delta>0}\bigcup_{\substack{{n}>^k {0} \\ {n}\in \mathbb{Z}^d}}  T_2^{-{n}}\times T_2^{-{n}} (U_{\delta}) $\vspace{.3cm}

\hspace{2cm}$\implies (h(x_1),h(x_2)) \in k-Prox(T_2)$. 

Similarly, using uniform continuity of $h^{-1}$, we can show the converse part.

   \textbf{(ii)} $(x_1,x_2)\in k-Asym(T_1)$
   
\hspace{2cm}$\implies \forall \epsilon>0 ~\exists {n}>^k{0} $ such that $d(T_1^{{n+m}}(x_1),T_1^{{n+m}}(x_2)) \leq \epsilon ~\forall {m}>^k{0}$
     
    By uniform continuity of $h$, for every $\delta>0$ there is an $\epsilon>0$ such that $d(T_1^{{n+m}}(x_1),T_1^{{n+m}}(x_2))< \epsilon$ implies that $d(h\circ T_1^{{n+m}}(x_1), h\circ T_1^{{n+m}}(x_2)) < \delta$.
    
\hspace{2cm}$\implies \forall \delta>0 ~\exists {n}>^k{0} $ such that $d(h\circ T_1^{{n+m}}(x_1), h\circ T_1^{{n+m}}(x_2)) \leq \delta ~\forall m>^k0 $ 

\hspace{2cm}$\implies \forall \delta>0 ~\exists {n}>^k{0} $ such that $d(T_2^{{n+m}}\circ h(x_1), T_2^{{n+m}}\circ h(x_2)) \leq \delta ~\forall {m}>^k{0} $

\hspace{2cm}$\implies (h(x_1),h(x_2)) \in k-Asym(T_2)$. 
    
    Similarly, using uniform continuity of $h^{-1}$, we can show the converse part.

    \textbf{(iii)} Let there exist an $\epsilon>0 $ such that $ y\not\in k-Asym_{\epsilon}(T_1)(x)$, i.e., $(x,y)\not\in k-Asym_{\epsilon}(T_1)$.
    Suppose for every $\delta>0, ~(h(x),h(y)) \in k-Asym_{\delta}(T_2)$.
    Then, for every $\epsilon>0, ~(x,y)\in k-Asym_{\epsilon}(T_1)$, which is a contradiction.
    Thus, there exists a $\delta>0$ such that $~(h(x),h(y))\not\in k-Asym_{\delta}(T_2)$, i.e., $ h(y)\not\in k-Asym_{\delta}(T_2)(h(x))$.
\end{proof}

\begin{theorem} Let $h$ be a conjugacy from $(X,T_1)$ to $(Y,T_2)$.
    
    (i) $(X,T_1)$ is $k-$type Li Yorke sensitive if and only if $(Y,T_2)$ is $k-$type Li Yorke sensitive.

    (ii) $(X,T_1)$ is equicontinuous if and only if $(X,T_2)$ is equicontinuous.
    
\end{theorem}
\begin{proof}
    \textbf{(i)} $(X,T_1)$ is $k-$type Li Yorke sensitive.
    
    $\Leftrightarrow$ $\exists \epsilon>0 $ such that $ ~~\forall x\in X$ there is a sequence $(x_n) \in k-Prox(T_1)(x) \setminus k-Asym_{\epsilon}(T_1)(x)$ with $(x_n) \rightarrow x$.
    
    $\Leftrightarrow$ $\exists \epsilon>0 $ such that $ ~~\forall x\in X$ there is a sequence $(x_n) \in k-Prox(T_1)(x)$ with $(x_n)\not\in k-Asym_{\epsilon}(T_1)(x)$ and $(x_n) \rightarrow x$.
    
    $\Leftrightarrow$ $\exists \delta>0 $ such that $~~\forall x\in X$ there is a sequence $(h(x_n)) \in k-Prox(T_2)(h(x))$ with $(h(x_n))\not\in k-Asym_{\delta}(T_2)(h(x))$ and $(h(x_n)) \rightarrow h(x)$.
    
    $\Leftrightarrow$ $\exists \delta>0 $ such that $~~\forall y\in Y$ there is a sequence $(y_n) \in k-Prox(T_2)(y)$ with $(y_n)\not\in k-Asym_{\delta}(T_2)(y)$ and $(y_n) \rightarrow y$.
    
    $\Leftrightarrow$ $~\exists \delta>0 $ such that $~~\forall y\in Y$ there is a sequence $(y_n) \in k-Prox(T_2)(y) \setminus k-Asym_{\delta}(T_2)(y)$ with $(y_n) \rightarrow y$.
    
    $\Leftrightarrow$ $(Y,T_2)$ is $k-$type Li Yorke sensitive.

    \textbf{(ii)} Let $(X,T_1)$ be equicontinuous, i.e.,
    for every $\epsilon>0$ there exists a $\delta>0$ such that  $~d(x_1,x_2)<\delta$ implies $d(T_1^{{n}}(x_1),T_1^{{n}}(x_2)<\epsilon$ for all ${n}\in\mathbb{Z}^d$.
    By uniform continuity of $h^{-1} $, for every $\delta>0$, there exists $\delta'>0$ such that $~d(h(x_1),h(x_2))<\delta'$ implies that $d(x_1,x_2)<\delta$. 
    By uniform continuity of $h$, for every $\epsilon'>0$, there exists $\epsilon>0$ such that $~d(T_1^{{n}}(x_1),T_1^{{n}}(x_2)<\epsilon$ implies that $d(h\circ T_1^{{n}}(x_1), h\circ T_1^{{n}}(x_2)<\epsilon'$ i.e., $d(T_2^{{n}}\circ h(x_1),T_2^{{n}}\circ h(x_2)<\epsilon'$.
   Thus, we have for every $\epsilon'>0$, a $\delta'>0$ such that  $~d(h(x_1),h(x_2))<\delta'$ implies that $d(T_2^{{n}}(h(x_1)),T_2^{{n}}(h(x_2))<\epsilon'$ for all ${n}\in\mathbb{Z}^d$. Hence $(X,T_2)$ is equicontinuous. Similarly, we can prove the converse part.
\end{proof}

\begin{definition}
    Two points $x,y\in X$ are said to form a \emph{$k-$type Li-Yorke pair} if $(x,y)\in k-Prox(T)\setminus k-Asym(T)$ and the set of all $k$-type Li-Yorke pairs is denoted by $k-LY(T)$. 
    A subset $S$ of $X$ is said to be a \emph{$k-$ scrambled set} if any two distinct pairs of points in $S$ form a $k-$type Li-Yorke pair.
    A dynamical system $(X,T)$ is said to be \emph{$k$-type Li-Yorke chaotic} if $X$ contains an uncountable scrambled set.
\end{definition}

\begin{theorem} Let $h$ be a conjugacy from $(X,T_1)$ to $(Y,T_2)$.
    
    (i) $(x_1,x_2)\in k-LY(T_1)$ if and only if $(h(x_1),h(x_2))\in k-LY(T_2)$.

    (ii) $(X,T_1)$ is $k-$type Li Yorke chaotic if and only if $(Y,T_2)$ is $k-$type Li Yorke chaotic.

\end{theorem}
\begin{proof}
    \textbf{(i)} $(x_1,x_2)\in k-LY(T_1)$
    
\hspace{1cm}$\Leftrightarrow (x_1,x_2)\in k-Prox(T_1)$ and $(x_1,x_2)\not\in k-Asym(T_1)$

\hspace{1cm}$\Leftrightarrow (h(x_1),h(x_2))\in k-Prox(T_2)$ and $(h(x_1),h(x_2))\not\in k-Asym(T_2)$

\hspace{1cm}$\Leftrightarrow (h(x_1),h(x_2))\in k-LY(T_2)$

    \textbf{(ii)} $(X,T_1)$ is $k-$type Li Yorke chaotic
    
\hspace{1cm}$\Leftrightarrow$ 
$\exists$ an uncountable set $S$ in $X$ such that $\forall x,y \in S, (x,y)\in k-LY(T_1)$
    
\hspace{1cm}$\Leftrightarrow$ 
$\exists$ an uncountable set $h(S)$ in $Y$ such that $\forall h(x),h(y) \in h(S), $ $ (h(x),h(y))\in k-LY(T_2)$
    
\hspace{1cm}$\Leftrightarrow (Y,T_2)$ is $k-$type Li Yorke chaotic.
\end{proof}

A continuous map $\pi: X\rightarrow Y$ is called an \textit{almost open map} if the image of any non-empty open set in $X$ has a non-empty interior in $Y$ or equivalently, if the preimage of any first category subset of $Y$ is a first category subset of $X$.

\begin{proposition}
    If $\pi:(X,T)\rightarrow (Y,S)$ is an almost open map with $(Y,S)$ being a $k-$type sensitive system, then $(X,T)$ is also $k-$type sensitive.
\end{proposition}
\begin{proof}
    Since $(Y,S)$ is $k-$type sensitive, by Theorem \ref{equivalences}, there exists $\epsilon>0$ such that $k-Asym_{\epsilon}(S)$ is a first category subset of $Y\times Y$. Since $\pi$ is almost open, so is $\pi\times\pi$ and thus $(\pi\times\pi)^{-1}(k-Asym_{\epsilon}(S))$ is a first category subset of $X\times X$. By uniform continuity of $\pi$, there is a $\delta>0$ such that $d(x_1,x_2)<\delta$ implies $d(\pi(x_1),\pi(x_2))<\epsilon$. Since $k-Asym_{\delta}(T) \subset (\pi\times\pi)^{-1}(k-Asym_{\epsilon}(S))$, it follows that $k-Asym_{\delta}(T)$ is a first category subset of $X\times X$. Hence $(X,T)$ is $k-$type sensitive again by Theorem \ref{equivalences}.
\end{proof}

\vspace{1cm}

\subsection{Induced $\mathbb{Z}^d$ actions}

We mentioned the works of Kamarudin and Dzul-Kifli \cite{kamarudin2021sufficient} in preliminaries. Following a similar approach, we have the following results.
As mentioned in the introduction, let $f$ be a homeomorphism of $X$, $r:\mathbb{Z}^d \rightarrow \mathbb{Z} $ be a homomorphism and $T_f$ be the induced $\mathbb{Z}^d$ action given by $T_f(\textbf{n},x)= f^{r(\textbf{n})}(x) $.

\begin{theorem}
     If $(X,f)$ is an equicontinuous dynamical system, then $(X,T_f)$ is an equicontinuous system.
\end{theorem}
\begin{proof}
    $(X,f)$ is equicontinuous, i.e., $\{f^n: n\in \mathbb{Z}\}$ is an equicontinuous family. Thus $\{ T_f^n:n\in \mathbb{Z}^d\} $ $=\{f^{r(n)}: n\in \mathbb{Z}^d\} $ $\subset\{f^n: n\in \mathbb{Z}\}$ is an equicontinuous family. Hence $(X, T_f)$ is an equicontinuous system. 
\end{proof}

\begin{proposition}
    If there exists $\textbf{m}>^k\textbf{0}$ such that $r(\textbf{m})=1$, then for any positive integer $n$, we have  $n\textbf{m}\in \mathbb{Z}^d$, $n\textbf{m}>^k\textbf{0}$ and $r(n\textbf{m})=n$.
\end{proposition}
\begin{proof}
     Suppose that there is an $\textbf{m}>^k\textbf{0}$ such that $r(\textbf{m})=1$ where $k\in \{1,2,3,\dots,2^d\}$. Clearly $n\textbf{m}>^k\textbf{0}$. Now,  $r(\textbf{m}) = h_1m_1 + h_2m_2 + \dots + h_dm_d = 1$ where $h_i \in \mathbb{Z}$ for all $i\in {1, 2, \dots, d}$ and $\textbf{m} = (m_1, m_2, \dots, m_d ) \in \mathbb{Z}^d$ . Then, for all $n \in \mathbb{N}$, $n = n(1) = n(h_1m_1 + h_2m_2 + \dots + h_dm_d ) = h_1(nm_1) + h_2(nm_2) + \dots + h_d (nm_d ) = r(n\textbf{m})$.
\end{proof}

\begin{theorem}
    Suppose $k\in \{1,2,3,\dots,2^d\}$ and there exists $\textbf{m}>^k\textbf{0}$ such that $r(\textbf{m})=1$.
  
        (i) If $f$ is sensitive, then $T_f$ is $k-$type sensitive. 
        
        (ii) If $x$ is a periodic point in $(X,f)$, then $x$ is a $k-$type periodic point in $(X,T_f)$.
         
\end{theorem}

\begin{proof}
     
    \textbf{(i)} Since $f$ is sensitive, there is a $\delta > 0$ such that for every $x\in X$ and $\epsilon>0$ there exists $y \in B_d(x,\epsilon)$ and $l>0$ such that $d(f^l(x),f^l(y))>\delta$. Choosing $\textbf{s} = l\textbf{m}$, we have $\textbf{s} >^k \textbf{0}$ and 
    $d(T_f^\textbf{s}(x),T_f^\textbf{s}(y)) = d(f^l(x),f^l(y))>\delta$.
    Hence $T_f$ is $k$ type sensitive.

    \textbf{(ii)}
    Choose a positive integer $n$ such that $f^n(x) = x$. Then $T_f^{n\textbf{m}}(x) = f^{r(n\textbf{m})}(x) = f^n(x) = x $. Hence $x$ is a $k-$type periodic point in $(X,T_f)$.
\end{proof}

\noindent The following corollary follows from the above theorem.

\begin{corollary}
    Suppose $k\in \{1,2,3,\dots,2^d\}$ and there exists $\textbf{m}>^k\textbf{0}$ such that $r(\textbf{m})=1$.
   
    If $(X,f)$ is Devaney (or Auslander-Yorke) chaotic, then $(X,T_f)$ is $k-$type Devaney (or $k-$type Auslander-Yorke) chaotic. 
\end{corollary}

\begin{theorem}
    Suppose $k\in \{1,2,3,\dots,2^d\}$ and there exists $\textbf{m}>^k\textbf{0}$ such that $r(\textbf{m})=1$.

    (i) If $(x,y)$ is a proximal pair in $(X,f)$, then $(x,y)$ is a $k-$type proximal pair in $(X,T_f)$. 
    
    (ii) If $y\not\in Asym_{\delta}(f)(x)$, then $y\not\in k-Asym_{\delta}(T_f)(x)$. Hence, $(x,y)\not\in Asym(f)$ implies that $(x,y)\not\in k-Asym(T_f)$.
    
    (iii) If $(X,f)$ is Li-Yorke sensitive, then $(X,T_f)$ is $k-$type Li-Yorke sensitive.
    
    (iv) If $(x,y)\in LY(f)$, then $(x,y)\in k-LY(T_f)$. Hence, if $(X,f)$ is Li-Yorke chaotic, then $(X,T_f)$ is $k-$type Li-Yorke chaotic.
    
    (v) If $(x,y)$ is a $k-$type asymptotic pair in $(X,T_f)$, then $(x,y)$ is an asymptotic pair in $(X,f)$.   
    
\end{theorem}
\begin{proof}
    \textbf{(i)} $(x,y)\in Prox(f)$ 
    
\hspace{1cm}$\implies (x,y) \in \bigcap_{\epsilon>0}\bigcup_{\substack{n> 0 \\ n\in \mathbb{Z}}} \left[ f^{-n}\times f^{-n} (V_{\epsilon}) \right]$ \vspace{0.3cm}

\hspace{1cm}$\implies \forall \epsilon>0 ~\exists n>0 
$ such that $ d(f^n(x),f^n(y)) <  \epsilon$
    
\hspace{1cm}$\implies \forall \epsilon>0 ~\exists \textbf{m}n>^k\textbf{0} $ such that $d(T_f^{\textbf{m}n}(x),T_f^{\textbf{m}n}(y)) <  \epsilon$
    
\hspace{1cm}$\implies (x,y) \in \bigcap_{\epsilon>0}\bigcup_{\substack{\textbf{n}>^k \textbf{0} \\ \textbf{n}\in \mathbb{Z}^d}} \left[ T_f^{\textbf{-n}}\times T_f^{\textbf{-n}} (V_{\epsilon}) \right]$ \vspace{0.3cm}
    
\hspace{1cm}$\implies (x,y) \in k-Prox(T_f).$

    \textbf{(ii)} If $y\not\in Asym_{\delta}(f)(x)$,
then $(x,y)\not\in Asym_{\delta}(f)$. Let $\textbf{n}>^k\textbf{0}$. If $r(\textbf{n})>0$, then choose $n'>0$ such that $d(f^{r(\textbf{n})+n'}(x), f^{r(\textbf{n})+n'}(y)) > \delta$. 
Defining $\textbf{n}'= n'\textbf{m}$, we see that $\textbf{n}'>^k\textbf{0}$ and $r(\textbf{n}+\textbf{n}')=r(\textbf{n})+n'$; thus $d(T_f^{\textbf{n}+\textbf{n}'}(x), T_f^{\textbf{n}+\textbf{n}'}(y)) > \delta$. 
In case $r(\textbf{n})\leq 0$, then choose $n>0$ such that $d(f^{n+1}(x), f^{n+1}(y)) > \delta$ and in this case, define $\textbf{n}'=(n+1-r(\textbf{n}))\textbf{m}$. 
Then $\textbf{n}'>^k\textbf{0}$ and $r(\textbf{n}+\textbf{n}')=n+1$. 
Thus, $d(T_f^{\textbf{n}+\textbf{n}'}(x), T_f^{\textbf{n}+\textbf{n}'}(y)) > \delta$. Hence $(x,y)\not\in k-Asym_{\delta}(T_f)$ and $y\not\in k-Asym_{\delta}(T_f)(x)$.

    \textbf{(iii)} $(X,f)$ is Li Yorke sensitive.
    
\hspace{1cm}$\implies \exists \delta>0 $ such that $\forall x\in X, ~x \in \overline{Prox(f)(x) \setminus Asym_{\epsilon}(f)(x)}$ 
    
\hspace{1cm}$\implies \exists \delta>0 $ such that $\forall x\in X,$ there is a sequence $(x_n) $ in $ Prox(f)(x) \setminus Asym_{\delta}(f)(x)$ with $(x_n)\rightarrow x.$ 
    
\hspace{1cm}$\implies \exists \delta>0 $ such that $\forall x\in X,$ there is a sequence $ (x_n)\in Prox(f)(x)$ with $(x_n) \not\in Asym_{\delta}(f)(x)$ and $(x_n)\rightarrow x.$ 
    
\hspace{1cm}$\implies \exists \delta>0 $ such that $\forall x\in X,$ there is a sequence $ (x_n)\in k-Prox(T_f)(x)$ with $(x_n) \not\in k-Asym_{\delta}(T_f)(x)$ and $(x_n)\rightarrow x.$
    
\hspace{1cm}$\implies \exists \delta>0 $ such that $\forall x\in X,$ there is a sequence $(x_n) $ in $ k-Prox(T_f)(x) \setminus k-Asym_{\delta}(T_f)(x)$ with $(x_n)\rightarrow x.$
    
\hspace{1cm}$\implies \exists \delta>0 $ such that $\forall x\in X, ~x \in \overline{k-Prox(T_f)(x) \setminus k-Asym_{\epsilon}(T_f)(x)}.$
    
\hspace{1cm}$\implies (X,T_f)$ is Li Yorke sensitive.

    \noindent\textbf{(iv)} $(x,y)\in LY(f)$ 
    
\hspace{1cm}$\implies$ $(x,y)\in Prox(f)$ and $(x,y)\not\in Asym(f)$ 
    
\hspace{1cm}$\implies$ $(x,y)\in k-Prox(T_f)$ and $(x,y)\not\in k-Asym(T_f)$ 
    
\hspace{1cm}$\implies$ $(x,y)\in k-LY(T_f)$.
    
    \noindent\textbf{(v)} 
Suppose $(x, y) \in k-Asym(T_f)$. 
This implies that for every $\epsilon > 0$, there exists $\mathbf{n} >^k \mathbf{0}$ such that $d(T_f^{\mathbf{n} + \mathbf{n'}}(x), T_f^{\mathbf{n} + \mathbf{n'}}(y)) \leq \epsilon$ for all $\mathbf{n'} >^k \mathbf{0}$. 
Now, for a given $\epsilon > 0$, let $n = max\{ r(\mathbf{n}), 1\}$. 
If $r(\mathbf{n}) > 0$, then for every $n' > 0$, we set $\mathbf{n}' = n' \mathbf{m}$. 
In the case where $r(\mathbf{n}) \leq 0$, for every $n' > 0$, we take $\mathbf{n'} = (1 + n' - r(\mathbf{n})) \mathbf{m}$. 
In both cases, we have $\mathbf{n'} >^k \mathbf{0}$, $r(\textbf{n}+\mathbf{n'})=n+n'$ and $d(f^{n+n'}(x), f^{n+n'}(y)) \leq \epsilon$ for every $n'>0$. 
This shows that $(x, y) \in Asym(f)$.\end{proof}
  
\textbf{Acknowledgements} The first author would like to acknowledge the financial support provided by the Council of Scientific and Industrial Research, India (CSIR), under fellowship file no: 09/1026(0044)/2021-EMR-I.

\end{document}